\documentclass{amsart}
\usepackage{graphicx}
\usepackage[all]{xy}
\usepackage{amscd}
\usepackage{amssymb}
\usepackage{amsmath}
\usepackage{amsthm}
\usepackage{amsfonts}
\usepackage{graphics}
\usepackage{epsfig,psfrag}
\usepackage[english]{babel}
\usepackage{latexsym}
\usepackage{mathrsfs}

\newtheorem{theorem}{Theorem}[section]

\newtheorem{lemma}[theorem]{Lemma}
\newtheorem{claim}{Claim}

\theoremstyle{definition}

\theoremstyle{remark}
\newtheorem{remark}[theorem]{Remark}

\begin{document}

\title[Comparing $h$-genera, Bridge-1 genera and Heegaard genera of knots]
{Comparing $h$-genera, bridge-1 genera and Heegaard genera of knots}

\author{Ruifeng Qiu}
\address{School of Mathematical Sciences, Key Laboratory of MEA(Ministry of Education) \& Shanghai Key Laboratory of PMMP, East China Normal University, Shanghai 200241, China}
\email{rfqiu@math.ecnu.edu.cn}

\author{Chao Wang}
\address{School of Mathematical Sciences, Key Laboratory of MEA(Ministry of Education) \& Shanghai Key Laboratory of PMMP, East China Normal University, Shanghai 200241, China}
\email{chao\_{}wang\_{}1987@126.com}

\author{Yanqing Zou}
\address{School of Mathematical Sciences, Key Laboratory of MEA(Ministry of Education) \& Shanghai Key Laboratory of PMMP, East China Normal University, Shanghai 200241, China}
\email{yqzou@math.ecnu.edu.cn}

\subjclass[2010]{Primary 57M25; Secondary 57N10}

\keywords{$h$-genus, Heegaard genus, bridge $b$ genus, tunnel number}

\thanks{The authors are supported by National Natural Science Foundation of China (NSFC), grant Nos. 12131009, 12371067, 12471065 and Science and Technology Commission of Shanghai Municipality (STCSM), grant No. 22DZ2229014.}

\begin{abstract}
Let $h(K)$, $g_H(K)$, $g_1(K)$, $t(K)$ be the $h$-genus, Heegaard genus, bridge-1 genus, tunnel number of a knot $K$ in the $3$-sphere $S^3$, respectively. It is known that $g_H(K)-1=t(K)\leq g_1(K)\leq h(K)\leq g_H(K)$. A natural question arises: when do these invariants become equal?

We provide the necessary and sufficient conditions for equality and use these to show that for each integer $n\geq 1$, the following three families of knots are infinite:
\begin{center}
 $A_{n}=\{K\mid t(K)=n<g_1(K)\}$, \\
 $B_{n}=\{K\mid g_1(K)=n<h(K)\}$, \\
$C_{n}=\{K\mid h(K)=n<g_H(K)\}$. \\
\end{center}
This result resolves a conjecture in \cite{Mo2}, confirming that each of these families is infinite.

\end{abstract}

\date{}
\maketitle


\section{Introduction}\label{sec:Intro}
It is well known that for any integer $g\geq 0$, $S^3$ contains a closed, orientable and separating genus $g$ surface $\Sigma_{g}$ so that 
 $\overline{S^3-\Sigma_{g}}$ is a union of two handlebodies $H$ and $H'$. Denote $H\cup_{\Sigma_{g}} H'$  a  Heegaard splitting of $S^3$. One great result by 
 Waldhausen \cite{wal} shows that for any integer $g\geq 0$, there is only one genus $g$ Heegaard splitting of $S^3$ up to isotopy. Hence, according to the position of a knot $K$ in $S^3$ with respect to $\Sigma_g$, one can define various knot invariants.

Classical invariants of this kind include the {\it bridge number} $b$ and the {\it Heegaard genus} $g_H$. For the bridge number, one requires that $K$ meets $\Sigma_0$ transversely and meets each side of $\Sigma_0$ in a trivial collection of $n$ arcs. Then, $b(K)$ is the minimal possible $n$. For the Heegaard genus, one requires that $K$ lies in one side of $\Sigma_g$ and is a core of the handlebody, or equivalently, $\Sigma_g$ also gives a Heegaard splitting of the complement of $K$. Then, $g_H(K)$ is the minimal possible $g$. We note that $g_H$ is related to the {\it tunnel number} $t$ of knots by the equation $g_H(K)=t(K)+1$.

The bridge number has a natural generalization for $\Sigma_g$ with a given $g\geq 0$. One similarly requires that $K$ intersects $\Sigma_g$ transversely and meets each side of $\Sigma_g$ in a collection of $m$ arcs that is parallel to $\Sigma_g$. Then, the {\it genus $g$ bridge number} $b_g(K)$ is the minimal possible $m$, and clearly $b_0=b$. For some interesting properties and results about the general $b_g$, one can see \cite{Do,TT,Zu}. The above $K$ is said to be in $m$ {\it bridge position} with respect to $\Sigma_g$. Since for a given number $m=n>0$ such an $n$ bridge position always exists for some $g\geq 0$, one can dually define the {\it bridge $n$ genus} $g_n(K)$ to be the minimal possible $g$. Then, note that $b_0(K)\geq b_1(K)\geq\cdots$ and $g_1(K)\geq g_2(K)\geq\cdots$ converge to $1$ and $0$, respectively. And $b_g(K)>n$ if and only if $g_n(K)>g$. The two sequences $\{b_g(K)\}_{g\geq 0}$ and $\{g_n(K)\}_{n\geq 1}$ are related by the equations $b_g(K)=|\{n\mid g_n(K)>g\}|+1$ and $g_n(K)=|\{g\mid b_g(K)>n\}|$.

In this paper, we explore the relationship between these genera and the $h$-genus of knots, as defined in \cite{Mo1}. Here, one requires that $K$ lies in $\Sigma_g$. Then, $h(K)$ is defined to be the minimal possible $g$. In \cite{Mo2,Mo3}, it was shown that the inequalities $t(K)\leq g_1(K)\leq h(K)\leq g_H(K)$ hold for any $K$. Additionally, it is not hard to see that $g_n(K)\leq g_{n+1}(K)+1$. Therefore, we obtain the sequence of inequalities:
\[\cdots\leq g_2(K)\leq g_1(K)\leq h(K)\leq g_H(K)\leq g_1(K)+1\leq g_2(K)+2\leq\cdots\quad (\ast)\]
where $g_H(K)\leq g_1(K)+1$ is equivalent to $t(K)\leq g_1(K)$.

For any Heegaard splitting $H\cup_{\Sigma_{g}} H'$ of $S^3$,  a pair of essential disks $(D\subset H, D'\subset H')$ is {\it complementary} if $\mid \partial D\cap \partial D'\mid =1$, and exactly  one of $D$ and $D'$ meets $K$.
After carefully examing those combinatorial structures of knots, we obtain the following necessary and sufficient conditions for equality among the genera.
\begin{theorem}
\label{theorem1}
  (I)  Given $i\geq 0$ and $l>0$, $g_i(K)=g_{i+l}(K)+l$ holds if and only if there exists a $\Sigma_g$ realizing $g_i(K)$ with $l$ complementary pairs $(D_j,D_j')$, $1\leq j\leq l$, where the $l$ bouquets of circles $\partial D_1\cup\partial D_1',\ldots,\partial D_l\cup\partial D_l'$ are pairwise disjoint. \\
  (II) The equality $g_H(K)=h(K)+1$ holds if and only if there exists a $\Sigma_g$ realizing $g_H(K)$ with a complementary pair $(D,D')$ and a spanning annulus $A$, so that exactly one of $D$ and $D'$ meets $A$.
\end{theorem}
In \cite{Mo2}, Morimoto considered the following three families of knots: 
\begin{align*}
A_n&=\{K\mid h(K)=n<g_H(K)\},\\
B_n&=\{K\mid g_1(K)=n<h(K)\},\\
C_n&=\{K\mid t(K)=n<g_1(K)\}.
\end{align*}
It is not hard to see that $A_{n}$, $B_{n}$ and $C_{n}$ form a partition of all nontrivial knots. In \cite{Mo2}, it was conjectured that $A_{n}$, $B_{n}$ and $C_{n}$ are all nonempty. 

Using Heegaard distance, mapping class group theory, and combinatorial techniques,  we obtain the following general result, which answers the above conjecture.

\begin{theorem}\label{thm:main}
For each integer $n>0$, the families $A_n$, $B_n$, $C_n$ are all infinite.
\end{theorem}

\begin{remark}
During the proof of    Theorem \ref{thm:main},  we find some interesting 
phenomenons as follows.
\begin{itemize}
     \item[$\bullet$] All torus knots belong to $A_{1}$, while $A_{n}$ contains most of all tunnel number $n$ cable knots.
    \item[$\bullet$] $B_{n}$ contains tunnel number $n$ knots of which complements admit locally large,  but distance two Heegaard splittings.
    \item[$\bullet$]  $C_{n}$ contains all tunnel number $n$,  high distance hyperbolic knots.
\end{itemize}
\end{remark}

The proof of Theorem 1 is divided into Lemmas 2.1 and 2.3, and the proof of Theorem 2 is presented in Sections 3, 4, and 5.

\section{The proof of Theorem~\ref{theorem1}}\label{sec:P}
Let $K$ be a knot in $S^3$, and let $\Sigma_g$ be a Heegaard surface of $S^3$, where $K$ meets $\Sigma_g$ transversely. Let $H$ and $H'$ be the two handlebodies bounded by $\Sigma_g$ in its two sides, respectively. Consider two proper disks $D\subset H$ and $D'\subset H'$ which meet $K$ transversely in at most one point. We call the disk pair $(D,D')$ {\it complementary} if  $\partial D$ and $\partial D'$ in $\Sigma_g$ meet transversely in exactly one point, and exactly one of $D$ and $D'$ meets $K$. We will regard $g_H$ as $g_0$ for convenience.

\begin{lemma}\label{lem:gbEq}
Given $i\geq 0$ and $l>0$, $g_i(K)=g_{i+l}(K)+l$ holds if and only if there exists a $\Sigma_g$ realizing $g_i(K)$ with $l$ complementary pairs $(D_j,D_j')$, $1\leq j\leq l$, where the $l$ bouquets of circles $\partial D_1\cup\partial D_1',\ldots,\partial D_l\cup\partial D_l'$ are pairwise disjoint.
\end{lemma}

\begin{proof}
Let $\Sigma_{g'}$ be a Heegaard surface realizing $g_{i+l}(K)$. Let $H_{g'}$ be a handlebody bounded by $\Sigma_{g'}$. Then $K\cap H_{g'}$ is a collection of $i+l$ arcs $\alpha_j$, $1\leq j\leq i+l$, which is parallel to $\Sigma_{g'}$. So, there are disjoint disks $D_j\subset H_{g'}$, $1\leq j\leq i+l$, so that each $\beta_j=D_j\cap\Sigma_{g'}$ is an arc and $\partial D_j=\alpha_j\cup\beta_j$. Let $N(\alpha_j)$ be a sufficiently small open regular neighborhood of $\alpha_j$ in $H_{g'}$ for $1\leq j\leq l$. Then, by removing all the $N(\alpha_j)$ from $H_{g'}$, we get a handlebody, whose boundary is a Heegaard surface with genus $g_{i+l}(K)+l$. We denote this surface by $\Sigma_g$.

If $i>0$, then those $i+l$ arcs in the other side of $\Sigma_{g'}$ become $i$ arcs lying in one side of $\Sigma_g$. Since $\alpha_j$ is a core of $N(\alpha_j)$, we can check that $K$ is in $i$ bridge position with respect to $\Sigma_g$. When $i=0$, we see that $K$ lies in one side of $\Sigma_{g}$ and it meets a proper disk in the handlebody transversely in exactly one point. Note that $K$ is parallel to $\Sigma_g$. So it is a core of the handlebody.

Hence, $g_i(K)\leq g_{i+l}(K)+l$ always holds. If the equality holds, then $\Sigma_g$ realizes $g_i(K)$. Note that the disk in the closure of $N(\alpha_j)$ dual to $\alpha_j$ together with the $D_j$ with the part in $N(\alpha_j)$ removed gives a complementary pair for $1\leq j\leq l$. Clearly the boundaries from different pairs are disjoint. So we have the ``only if'' part.

Conversely, given the $\Sigma_g$ and the complementary pairs, we can assume that the disks from different pairs are disjoint. Otherwise we can assume that $D_j$ meets $D_k$ transversely for some $j\neq k$, where the intersection circles do not meet $K$. We can then pick an innermost circle $\gamma$ in $D_j$, which bounds a disk $D\subset D_j$. Now, since $\gamma$ also bounds a disk $D'\subset D_k$, where $D'$ meets $K$ if and only if $D$ meets $K$, one can modify $D_k$ by a surgery along $D$. So, we have new complementary pairs satisfying the condition, and $\gamma$ is removed. By this method, we can remove all those possible intersections between $D_1,\ldots,D_l$ and between $D_1',\ldots,D_l'$.

Now assume that $D_1$ meets $K$ and $D_1'$ does not meet $K$. Let $N(D_1)$ be a small open regular neighborhood of $D_1$ in $H$. Then, by removing $N(D_1)$ from $H$, we get a handlebody, whose boundary is a Heegaard surface with genus $g_i(K)-1$. Let $\alpha$ be the subarc of $K$ lying in $N(D_1)$. Then, $\alpha$ can be moved into the new Heegaard surface via $D_1'$. Denote this new surface by $\Sigma$. We need to show that $K$ is in $i+1$ bridge position with respect to $\Sigma$.

If $i=0$, then $K$ is a core of $H$. There exist a simple closed curve $\beta$ in $\Sigma_g$ and a spanning annulus $A$ in $H$ so that $\partial A=K\cup\beta$ and $A$ intersects $D_1$ transversely in some circles and arcs. Note that there is exactly one arc containing $K\cap D_1$ and it meets $\beta$. So the arc is a spanning arc in $A$. Then we see that any innermost circle in $D_1$ bounds a disk in $A$, and so can be removed by a surgery. If there is no such circle, then any outermost arc in $D_1$ can be removed by a surgery, since it cuts off a disk in $A$. Hence we can modify $A$ so that it meets $D_1$ in exactly one arc. Then, we see that the arc $K\setminus\alpha$ can also be moved into $\Sigma$.

If $i>0$, then $K\cap H$ consists of $i$ arcs, denoted by $\alpha_{l+1},\ldots,\alpha_{l+i}$. As before, we have disjoint disks $D_j\subset H$, $l<j\leq l+i$, where $\partial D_j=\alpha_j\cup\beta_j$ and $\beta_j=D_j\cap\Sigma_g$. We can require that the union of those disks meets $D_1$ transversely in some circles and arcs. Then, similar to the case of $i=0$, we can modify the disks so that there is exactly one arc in the intersection. The arc divides some $D_j$ into two disks, and we see that $(K\cap H)\setminus\alpha$ is parallel to $\Sigma$. We also have $i$ arcs in $K\cap H'$, and there are disjoint disks $D_j'\subset H'$, $l<j\leq l+i$, as before. Now, one can modify the disks so that they do not meet $D_1'$ and each $\partial D_j'$ meets $\partial D_1$ transversely. Then, we can move those disks across $D_1'$ so that their boundaries lie in $\Sigma$. We see that the disk corresponding to $\alpha$ can be chosen to avoid the $D_j'$, $l<j\leq l+i$.

So, in each case, $K$ is in $i+1$ bridge position with respect to $\Sigma$. Then, consider the pairs $(D_2,D_2'),\ldots,(D_l,D_l')$. By induction, one can get a Heegaard surface $\Sigma_{g'}$ with genus $g_i(K)-l$, and $K$ is in $i+l$ bridge position with respect to $\Sigma_{g'}$. Hence, $g_{i+l}(K)\leq g_i(K)-l$. Then $g_i(K)=g_{i+l}(K)+l$, and we get the ``if'' part.
\end{proof}

\begin{remark}\label{rem:g0Eq}
Let $i=0$, $l=1$. We see that $g_H(K)=g_1(K)+1$ holds if and only if there exists a $\Sigma_g$ realizing $g_H(K)$ with a complementary pair $(D,D')$.
\end{remark}

\begin{lemma}\label{lem:h0Eq}
The equality $g_H(K)=h(K)+1$ holds if and only if there exists a $\Sigma_g$ realizing $g_H(K)$ with a complementary pair $(D,D')$ and a spanning annulus $A$, so that exactly one of $D$ and $D'$ meets $A$.
\end{lemma}

\begin{proof}
Let $\Sigma_{g'}$ be a Heegaard surface realizing $h(K)$, and let $\Sigma$ be a parallel copy of $\Sigma_{g'}$. So, between the two Heegaard surfaces we have a product $\Sigma\times[0,1]$, where we identify $\Sigma$ and $\Sigma_{g'}$ with $\Sigma\times\{0\}$ and $\Sigma\times\{1\}$, respectively. Let $H_{g'}$ denote the handlebody bounded by $\Sigma$ that contains $\Sigma_{g'}$. Let $\gamma\subset\Sigma$ be the copy of $K$, and let $\delta\subset\gamma$ be an arc. Now modify $H_{g'}$ by removing a small open regular neighborhood of the disk $\delta\times[0,1]$ from it. Then, $\partial H_{g'}$ becomes a Heegaard surface so that $K$ is in $1$ bridge position with respect to it.

Let $\alpha$ be the arc $K\cap H_{g'}$. Let $N(\alpha)$ be a small open regular neighborhood of $\alpha$ in $H_{g'}$. As in the proof of Lemma~\ref{lem:gbEq}, by removing $N(\alpha)$ from $H_{g'}$, we can obtain a Heegaard surface $\Sigma_g$ with genus $h(K)+1$. Moreover, $K$ becomes a core lying in one side of $\Sigma_g$. Hence $g_H(K)\leq h(K)+1$ always holds. If the equality holds, then $\Sigma_g$ realizes $g_H(K)$. Note that $(\gamma\times[0,1])\cap H_{g'}$ gives a disk $D$, and the disk in the closure of $N(\alpha)$ dual to $\alpha$ together with $D\setminus N(\alpha)$ gives a complementary pair. In the handlebody bounded by $\Sigma_g$ that contains $K$, we also have a spanning annulus $A$ with $A\cap(\Sigma\times[0,1])=K$. Since $A\cap D\subset N(\alpha)$, we have the ``only if'' part.

Conversely, given the $\Sigma_g$ with the disks $(D,D')$ and annulus $A$, one can assume that $K$ is a core of $H$, $A$ intersects $D$ transversely in some circles and arcs, and $A$ does not meet $D'$. Then, as in the proof of Lemma~\ref{lem:gbEq}, one can remove the circles in $D$ by surgeries. Let $\eta$ be the arc in $D$ containing $K\cap D$. If an outermost arc in $D$ does not separate $\eta$ and the point $D'\cap D$, then it can be removed by a surgery as before. So we assume that there is no such arc. Let $\omega\subset D$ be an outermost arc separating $\eta$ and $D'\cap D$. Consider a proper annulus $A'$ in $H$ that cuts off a small regular neighborhood of $\partial D'$ in $H$. Then $A'\cap A=\emptyset$ and $A'\cap D$ is an arc parallel to $\omega$ in $D$. Let $D_0\subset D$ be the disk lying between $\omega$ and $A'\cap D$. Then, we can do a surgery to $A\cup A'$ along $D_0$. It turns $A$ into another spanning annulus satisfying the condition, and $\omega$ is removed. By the method, one can remove those arcs other than $\eta$. Hence $A$ can be modified so that $A\cap D=\eta$.

 Now, let $N(D)$ be a small open regular neighborhood of $D$ in $H$. As before, by removing $N(D)$ from $H$, we get a handlebody. Its boundary is a Heegaard surface with genus $g_H(K)-1$. We denote it by $\Sigma$. Let $\alpha$ be the arc $K\cap N(D)$. It can be moved into $\Sigma$ via $D'$. Since $A\cap D=\eta$, $K\setminus\alpha$ can also be moved into $\Sigma$. In $\Sigma$, the two arcs do not intersect, because $A\cap D'=\emptyset$. So $K$ can be moved into $\Sigma$. Hence, $h(K)\leq g_H(K)-1$. Then $g_H(K)=h(K)+1$, and we get the ``if'' part.
\end{proof}

\begin{remark}
The above method to simplify $A\cap D$ can also be used to replace the conditions in Lemma~\ref{lem:gbEq} by some weaker ones. For example, the conditions about those $l$ bouquets of circles can be replaced by $\partial D_i\cap\partial D_j'=\emptyset$, $\partial D_i'\cap\partial D_j'=\emptyset$, and if both $D_i$ and $D_j$ meet $K$, then $\partial D_i\cap\partial D_j=\emptyset$, where $1\leq i,j\leq l$ and $i\neq j$.
\end{remark}

\begin{figure}[h]
\includegraphics{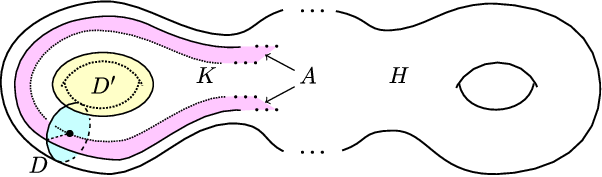}
\caption{The disk pair $(D,D')$ and the annulus $A$.}\label{fig:Lem2}
\end{figure}

\section{$A_{n}$ is infinite. }
When $n=1$, all torus knots  are contained in $A_{1}$. So $A_{1}$ is infinite. Therefore we assume that $n\geq 2$. It is well known that there is a nontrivial knot $K'$ with $n$ as its minimal Heegaard genus, i.e., $g_{H}(K')=n$, see \cite{JMM} for detail.  Let $K$ be a nontrivial knot lying in the torus boundary of  the closed neighborhood of $K'$, i.e.,  $K$ is  a cable knot with $K'$ as its companion. On one hand, by Theorem 1.2 in \cite{WZ}, except finitely many cases,  $t(K)=t(K')+1$. Hence $g_{H}(K)=g_{H}(K')+1=n+1$ and $h(K)\geq n$.  On the other hand, since 
$E(K')$, the complement of $K'$ in $S^3$, admits a genus $g_{H}(K')$ Heegaard splitting $V\cup_{S} W$. Without loss of generality, we assume that the torus boundary $\partial E(K')$ lies in the compression body $V$. By definition,  $V$ is the disk sum of a handlebody and $\partial E(K')\times I$. So  $K\subset \partial E(K')$ can be isotoped into  $\partial_{+}V=S$.  Hence $h(K)\leq n$. 

In all, $h(K)=n$ while $g_{H}(K)=n+1$. So $K\in A_{n}$. The abundance of choices of $K$ implies that $A_{n}$ is infinite. 

\section{$B_{n}$ is infinite.}
 For any positive integer $n$, let $H\cup_{S} H'$ be a genus $n+1$ Heegaard splitting of $S^3$. 
We choose a pair of complementary disks $D\subset H$ and $D'\subset H'$, i.e., $\mid \partial D\cap \partial D'\mid=1$. Let $K_{0}$ be the unknot isotopic to $\partial D'$ in $H$.  Denote $V$ the compression body $H-N(K_{0})$. In the following argument, we use the similar idea in building locally complicated Heegaard splittings in \cite{QZ,DZ}.

In the proof of Theorem 3.1 in \cite{MMS},  there is an essential simple closed curve 
$\alpha\subset S_{D}=\overline{S-\partial D}$, which bounds an essential disk $D_{\alpha}\subset H$ so that (1) $K_{0}$ passes through $D_{\alpha}$  in two times; (2) $\alpha$ intersects all disk boundaries  of ${V}$; (3) $\alpha$ intersects all disk boundaries of $H'$.

It is not hard to see that there is a pair of essential simple closed curves $\beta$ and $\gamma$ so that 
(1) both $\beta$ and $\gamma$ bound disks in $H$; (2) $\beta\cup \gamma$ fills $S_{D}$.
Let $\Psi_{N}=T_{\alpha}^N \circ T_{\beta}\circ T_{\gamma}^{-1} \circ T_{\alpha}^{-N}$, where $T_{\alpha}$ is a Dehn twist along $\alpha$. Then $\Psi_{N}$ maps $H$ back to $H$, and its restriction on $\partial H$ is reducible. Since $\beta$ and $\gamma$ is a pair of filling curves in $S_{D}$, $T_{\beta}\circ T_{\gamma}^{-1}$ is a pseudo-anosov map on $S_{D}$. Then there is a pair of stable and unstable measured laminations $\mu^{+}$ and $\mu^{-}$ of $T_{\beta}\circ T_{\gamma}^{-1}$ in the projective measured lamination space $\mathcal{PML}(S_{D})$. So $T_{\alpha}^N(\mu^+), T_{\alpha}^{N}(\mu^-)$ are stable and unstable laminations of $\Psi_{N}$.  Moreover, when $N$ goes to infinity, $T_{\alpha}^{N}(\mu^{+})$( resp. $T_{\alpha}^{N}(\mu^{-})$) converges to $\alpha$ in $\mathcal {PML}(S_{D})$.

By the choice of $\alpha$, there exists an $N>0$ so that
$T_{\alpha}^{N}(\mu^-)$ intersects all disk boundaries of $V$ in $S_{D}$.  Similarly, $T_{\alpha}^{N}(\mu^+)$ intersects all disk boundaries  of $H'$ in $S_{D}$. Hence
$T_{\alpha}^{N}(\mu^-)\notin \overline{\mathcal {D}_{S_{D}}(V)}$ ( resp. $T_{\alpha}^{N}(\mu^+)\notin \overline{\mathcal {D}_{S_{D}}(H')}$), the closure of all disk boundary curves in $\mathcal{PML}(S_{D})$.
 We use $\mathcal {C}(S)$ to represent the curve complex of $S$ and $\mathcal {AC}(S_{D})$ the arc and curve complex of $S_{D}$. For any essential simple closed curve $c\subset S$, $\pi_{S_{D}}(c)$ represents the subsurface projection of $c$ in $S_{D}$, see \cite{MM}.
Denote $\mathcal {D}_{S_{D}}(H')$ the disk complex of $H'$ spanned by all disk boundaries in $S_{D}$. Similar for $\mathcal {D}_{S_{D}}(V)$.
  By  a  similar argument in the proof of Theorem 2.7 in \cite{Hem}, there exists an $n'$ so that the distance 
 \begin{equation}
 \label{equation 2.1}
     d_{\mathcal {AC}(S_{D})}(\mathcal {D}_{S_{D}}(H'), \Psi_{N}^{n'}(\mathcal {D}_{S_{D}}(V)))\geq 2n+14, \text{where} ~\Psi_{N}^{n'} \text{ is the $n'$ power of $\Psi_{N}$}.
 \end{equation}
Denote $K_{n}$  the image $\Psi_{N}^{n'}(K_{0})$.  Next we will prove that 
$g_{H}(K_{n})=n+1$. 

\begin{theorem}
\label{Claim2.1}
(1) $g_{H}(K_{n})=n+1$ and $g_{1}(K_{n})=n$;\\
(2) There are at most two genera $n+1$ nonisotopic Heegaard splittings of $E(K_{n})$ up to isotopy.
\end{theorem}
\begin{proof} It is not hard to see that the boundary curve of each essential disk of $H'$ ( resp. $\Psi_{N}^{n'}(V)$) intersects 
$S_{D}$ essentially, i.e., $S_{D}$ is a compressible hole for  both $\mathcal {D}(H')$ and $\mathcal {D}(\Psi_{N}^{n'}(V))$. 
\begin{lemma} [Lemma 11.7 \cite{MS}]
\label{lemme2.1}
   Assume that $S_{D}$ is a compressible hole for both $H'$ and $\Psi_{N}^{n'}(V)$. For any essential disk $D'\subset H'$( resp. $D_{V}\subset \Psi_{N}^{n'}(V)$,  there is an essential disk $D''$  (resp. $D_{V}'$) with its boundary curve in $S_{D}$ so that 
\begin{eqnarray}
    diam_{\mathcal {AC}(S_{D})}(\pi_{S_{D}}(\partial D'), \partial D'')\leq 4;\\
    diam_{\mathcal {AC}(S_{D})}(\pi_{S_{D}}(\partial D_{V}), \partial D_{V}')\leq 4.
\end{eqnarray} 
\end{lemma} 
Therefore, by inequalities (1),(2) and (3),
\[ d_{\mathcal {AC}(S_{D})}(\pi_{S_D}\mathcal  {D}(H'), \pi_{S_{D}}\mathcal {D}(\Psi_{N}^{n'}(V)))\geq 2n+6.\]

Let $\Sigma$ be a minimal genus Heegaard surface of $E(K_{n})$. Then $g(\Sigma)\leq g(S)=n+1$. Since $S_{D}$ is an essential subsurface of $S$,  Theorem 1.1 in \cite{JMM} says that " Assume that $S$ and $\Sigma$ are two Heegaard surfaces of $E(K_{n})$ and $S_{D}\subset S$ is an essential subsurface. If  $d_{S_D}(S) > 2g(\Sigma)+2$,  then up to ambient isotopy,  $S\cap  \Sigma$ contains $S_{D}$."  In their setting,  $d_{S_{D}}(S)=d_{\mathcal {AC}(S_{D})}(\pi_{S_{D}}\mathcal {D}(H'), \pi_{S_{D}}\mathcal {D}(\Psi_{N}^{n'}(V)))\geq 2n+6> 2g(\Sigma)+2$. So  $S_{D}\subset \Sigma$.  
Remember $S_{D}$ is the closure of $S-\partial D$. Then $\chi(S_{D})=\chi(S)=2-2(n+1)$. Hence $\chi(\Sigma)\leq \chi(S_{D})\leq 2-2(n+1)$ and so $g(\Sigma)\geq n+1$. It means that $g(\Sigma)=n+1$. By Remark 2.2,  $g_{1}(K_{n})=n$.

By the above argument, we can isotope $\Sigma$ so that $\Sigma\cap S$ contains $S_{D}$ and $A=\Sigma-S_{D}$ is an annulus.
We claim that the interior of $A$ is fully contained in $\Psi_{N}^{n'}(V)$. For if not, then $A\cap H'$ contains an essential annulus $A'$ in $H'$. But, since  $\partial A'$ is isotopic to $\partial D$, $\partial A'$ intersects $\partial D'$ in one point. So $A'$ is boundary parallel to $H'$. A contradiction.  

So either $\Sigma$ is isotopic to $S$. Or, $A$ cuts out $T^2\times I$ from $\Psi_{N}^{n'}(V)$.  Hence for the Heegaard surface $\Sigma$, it cuts $E(K_{n})$ into one handlebody
$\overline{\Psi_{N}^{n'}(V)-T^{2}\times I}$ and $H'\cup T^{2}\times I$.  Here, $H'\cup T^{2}\times I$ is a  compression body for $\partial D$ is primitive in $H'$. 
\end{proof}

In the left part, we will prove that $h(K_{n})=g_{H}(K_{n})$. By Lemma \ref{lem:h0Eq}, it suffices to prove that in $\Psi_{N}^{n'}(V)\cup_{S} H'$,  there is no spanning annulus  $A\subset \Psi_{N}^{n'}(V)$ disjoint from  any essential disk $D'\subset H'$, i.e., $\partial A\cap \partial D'\neq \emptyset$. 

Suppose that there is a spanning annulus $A$ and an essential disk $D'\subset H'$ so that $K_{n}\subset \partial A$ and $\partial A\cap \partial D'=\emptyset$.  On one hand, $d_{\mathcal {AC}(S_{D})}(\pi_{S_{D}}(\partial A), \pi_{S_{D}}(\partial D'))\leq 2$. Since $S_{D}$ is a compressible hole for $\mathcal {D}(H')$,  by Lemma \ref{lemme2.1}, there is an  essential disk $D_{H'}\subset H'$ with its boundary in $S_{D}$ so that  
$diam_{\mathcal {A}{C}(S_{D})}(\pi_{S_{D}}(\partial D'), \partial D_{H'}')\leq 4$. Hence 
\begin{eqnarray}
\label{eqn4}
d_{\mathcal {AC}(S_{D})}(\pi_{S_{D}}(\partial A), \partial D_{H'})\leq 6.
\end{eqnarray}On the other hand,  for $A\subset 
\Psi_{N}^{n'}(V)$ is a spanning annulus, there is an essential disk $D''$ disjoint from $A$. 
Therefore,  $d_{\mathcal {AC}(S_{D})}(\pi_{S_{D}}(\partial A), \pi_{S_{D}}(\partial D''))\leq 2$. Since $S_{D}$ is also a compressible hole for $\mathcal {D}(\Psi_{N}^{n'}(V))$, by Lemma \ref{lemme2.1} again,  there is an essential disk $D'''$ with $\partial D'''\subset S_{D}$ so that 
\begin{eqnarray}
\label{eqn5}
    d_{\mathcal {AC}(S_{D})}(\pi_{S_{D}}(\partial A), \partial D''')\leq 6.
\end{eqnarray} 
By inequalities (\ref{eqn4}) and (\ref{eqn5}), $d_{\mathcal {AC}(S_{D})}(\partial D''', \partial D_{H'})\leq 12,$  against inequality (1). 

 The argument for $\overline{\Psi_{N}^{n'}(V)-T^{2}\times I}\cup_{\Sigma} (H'\cup T^{2}\times I )$ is similar. So we omit it.

\section{$C_{n}$ is infinite.}
For any positive integer $n$, let $K$ be a high distance knot with $t(K)=n$, equivalently,  there is a distance at least $2t(K)+5$ Heegaard splitting  $V\cup_{S} W$ for $E(K)$, see \cite{JMM} for detail. On the other hand, $E(K)$ admits a genus $g_{1}(K)+1$ Heegaard splitting $V'\cup_{S'} W'$ obtained from its $(g_{1}(K),1)$-bridge position. The inequality $(\ast)$ implies that $g(S')=g_{1}(K)+1\leq g(S)+1$. 
\begin{claim}
\label{claim2.2}
    $V'\cup_{S'} W'$ is either stabilized or boundary stabilized.
\end{claim}     
\begin{proof}
    Suppose not. Then $V'\cup_{S'} W'$ is neither stabilized nor boundary stabilized.  According to Scharlemann and Tomova' main result \cite{ST},  either\\
    (1) $V'\cup_{S'} W'$ is isotopic to $V\cup_{S} W$.  By the construction of $S'$, there is a spanning annulus $A\subset V'$ and an essential disk $E\subset W'$ so that $\partial A\cap \partial E$ is a point in $S'$. It means that Heegaard distance $d(S')\leq 2$. Since $S'$ is isotopic to $S$, $d(S)\leq 2$, which contradicts  the assumption that $d(S)\geq 2t(K)+5>5$. 
    Or, 
    (2) The Heegaard distance $d(S)\leq 2g(S')\leq 2g(S)+2=2t(K)+4$. It contradicts to the choice of $K$.  So this is impossible.
\end{proof}
By  Claim \ref{claim2.2}, there is a genus $g(S')-1$ Heegaard splitting for $E(K)$ obtained from a destabilization or boundary destabilization of $V'\cup_{S'} W'$. Hence $g_{H}(K)\leq g(S')-1=g_{1}(K)$. The inequality 
$(\ast)$ implies that $g_{1}(K)=g_{H}(K)$.  Hence $t(K)=n<g_{1}(K)$.


\newpage

\bibliographystyle{amsalpha}

\end{document}